\documentclass[11pt,reqno]{amsart}

\usepackage{latexsym}
\usepackage{amsfonts}
\usepackage{amsmath}
\usepackage{url}
\usepackage{graphicx}
\usepackage{hyperref}
\usepackage{amssymb}
\usepackage{amsmath}

\def\bl{\begin{lemma}}
\def\el{\end{lemma}}
\def\bth{\begin{theorem}}
\def\eth{\end{theorem}}
\def\bc{\begin{corollary}}
\def\ec{\end{corollary}}
\def\bcj{\begin{conjecture}}
\def\ecj{\end{conjecture}}
\def\bpr{\begin{proposition}}
\def\epr{\end{proposition}}
\def\bde{\begin{definition}}
\def\ede{\end{definition}}
\def\E{\mathbb{E}}
\def\H{\mathbb{H}}
\def\Pr{\mathbb{P}}

\newcommand{\dist}{\mbox{\rm dist}}
\newcommand{\be}{\begin{eqnarray}}
\newcommand{\ee}{\end{eqnarray}}
\newcommand{\eps}{{\mbox{$\epsilon$}}}
\newcommand{\R}{{\mathbb R}}

\newcommand{\Z}{{\mathbb Z}}

\newcommand{\C}{{\mathbb C}}

\renewcommand{\and}{\hbox{ {\rm and} }}

\newcommand{\Pp}{\mathcal P}

\newcommand{\mathcalG}{{\mathcal G}_*}
\newcommand{\mathcalGG}{{\mathcal G}_{**}}

\newcommand{\Isom}{{\rm Isom}}
\def\calD{{\mathcal D}}
\newcommand{\calV}{{\mathcal V}}

\newcommand{\G}{{\tt G}}

\newtheorem{theorem}{Theorem}[section]
\newtheorem{definition}{Definition}[section]
\newtheorem{lemma}[theorem]{Lemma}

\newtheorem{corollary}[theorem]{Corollary}
\newtheorem{proposition}[theorem]{Proposition}
\newtheorem{conjecture}[theorem]{Conjecture}
\newtheorem{quest}[theorem]{Question}

\theoremstyle{definition}
\newtheorem{example}[theorem]{Example}
\newtheorem{remark}[theorem]{Remark}
\numberwithin{equation}{section}

\title{Invariant embeddings of unimodular random planar graphs}
\author{Itai Benjamini and \'Ad\'am  Tim\'ar }
\begin{document}

\maketitle
\let\thefootnote\relax\footnotetext{\footnotesize{
\it{2010 Mathematics Subject Classification.} {\rm Primary 60D05. Secondary 60K99.}
}}
\begin{abstract}
Consider an ergodic unimodular random one-ended planar graph $\G$ of finite expected degree. We 
prove that it has an isometry-invariant locally finite embedding in the Euclidean plane if and only if it is invariantly amenable. By ``locally finite'' we mean that any bounded open set intersects finitely many embedded edges. In particular, there exist invariant embeddings in the Euclidean plane for the Uniform Infinite Planar Triangulation and for the critical Augmented Galton-Watson Tree conditioned to survive. Roughly speaking, a unimodular embedding of $\G$ is one that is jointly unimodular with $\G$ when viewed as a decoration. We show that $\G$ has a unimodular embedding 
in the hyperbolic plane if it is invariantly nonamenable, and it has a unimodular embedding 
in the Euclidean plane if and only if it is invariantly amenable. Similar claims hold for representations by tilings instead of embeddings. 
\end{abstract}

\section{Introduction}\label{intro}

\subsection{Main results}

Homogeneous (say, vertex-transitive) tilings of the Euclidean and hyperbolic planes are well-understood classical objects. Here we study {\it random} planar tilings that are ``homogeneous'' in some sense: they have an isometry-invariant law or satisfy a certain stationarity property, called unimodularity (see the definitions below). We are asking the question: when does a random infinite graph that is known to be almost surely planar have such a ``homogeneous'' embedding into the Euclidean or hyperbolic plane without accumulation points (i.e., a  {\it locally finite} embedding)? 

After formalizing the above question, one finds right away that a necessary condition is that the random graph with a properly chosen root is unimodular itself, so from now on, we are only interested in this family of random rooted graphs. Unimodularity, embedding and representation by a tiling will be defined later in the introduction. We say that the random rooted graph $(\G,o)$ has {\it finite expected degree} if the expected degree of $o$ is finite.

\begin{theorem}\label{uj}
An ergodic unimodular random one-ended planar graph $\G$ of finite expected degree has an isometry-invariant locally finite embedding 
into the Euclidean plane if and only if $\G$ is invariantly amenable. The same condition is necessary and sufficient to represent $\G$ by an isometry-invariant locally finite tiling.
\end{theorem}

One source of interest in invariant random embeddings of unimodular random graphs is examples such as the Uniform Infinite Planar Triangulation (UIPT), where invariant embeddings of some control of the edge length distribution would have some far-reaching consequencess. Our result can be seen as a step in this direction, since Theorem \ref{uj} applies for the UIPT, as stated in Corollary \ref{UIPT}.


When $\G$ has an embedding into the Euclidean or hyperbolic plane, the relative location of the embedded vertices and edges by this embedding from the viewpoint of each vertex can be used to decorate the vertices. If the decorated graph is still unimodular, we call the embedding unimodular. The more precise definition is given in the next subsection.


\begin{theorem}\label{dichotomy0}
An ergodic unimodular random one-ended planar graph $\G$ of finite expected degree has a unimodular locally finite embedding
\begin{itemize}
\item into the Euclidean plane if and only if $\G$ is invariantly amenable,
\item into the hyperbolic plane if $\G$ is invariantly nonamenable.
\end{itemize}
\end{theorem}

One can construct the embedding so that every edge is mapped into a broken line segment (piece-wise geodesic curve).
We mention that here and in the next theorem, the ``only if" part is missing from the second claim because of examples such as Example \ref{example_strict}. With some extra condition this could be ruled out and have a full characterization in the theorems; see the discussion after the example. 

\begin{theorem}\label{dichotomy}
An ergodic unimodular random one-ended planar graph $\G$ of finite expected degree can be represented by a unimodular locally finite tiling 
\begin{itemize}
\item in the Euclidean plane if and only if $\G$ is invariantly amenable, 
\item in the hyperbolic plane if $\G$ is invariantly nonamenable.
\end{itemize}
\end{theorem}
The tilings guaranteed by the theorem are such that the expected area of the tile containing the origin is finite.
Similarly to embeddings, we say that a tiling is locally finite if every bounded open subset of the plane intersects finitely many tiles. 
We will give precise definitions later in this section. We mention that the tiles in the above theorem can be required to be bounded polygons.

Theorems \ref{dichotomy0} and \ref{dichotomy} provide essentially complete dichotomic descriptions for the one-ended case. The cases not covered are those of $\G$ with 2 or infinitely many ends, in which situation a graph may or may not have any invariant embedding into one of the Euclidean or the hyperbolic plane (Remark \ref{infinitely_ended}). We do not treat 2 or infinitely many ends here, to avoid technical distractions from our main point.

In the case of invariantly amenable graphs, we will first construct an invariant embedding into the Euclidean plane, starting from a suitable invariant point process as the vertex set. This invariant embedded graph automatically defines a unimodular embedding. On the other hand, for invariantly nonamenable graphs, a unimodular embedding into the hyperbolic plane will be constructed directly, via circle packings. (This circle packing embedding would not work in the Euclidean case,
because Euclidean scalings provide an extra non-compact degree of
freedom,  making it unclear how to achieve unimodularity.) The intuitive claim that such a unimodular embedding is the ``Palm version'' of an invariant embedding 
does not seem to have been established in the hyperbolic setup; a similar statement in the Euclidean case is in the focus of \cite{HL}. This is the reason for the asymmetry in the Euclidean and hyperbolic cases of the above theorems. 
A new preprint by the second author and L\'aszl\'o T\'oth \cite{TT} settles the question of invariant embedding of nonamenable graphs into the hyperbolic plane, together with the cases of 2 and infinitely many ended graphs, left open by the present paper.

\subsection{Definitions}

Our focus is on Euclidean and hyperbolic spaces, hence the definitions will be phrased in this setting. One could ask questions in greater generality, for example by taking Lie groups as underlying spaces. We mention a few such directions in the concluding Section \ref{concluding}.

Without loss of generality {\it from now on we assume that $\G$ is a simple graph}, that is, it has no loop-edges or parallel edges. We also assume that {\it all the degrees are finite} in $\G$.

Next we define unimodularity. First we are using random walks, as this seems to be more natural to describe when unimodular random embeddings are considered.
Let $\mathcalG$ be the collection of all locally finite connected rooted graphs up to rooted isomorphism, and let $\mathcalGG$ be the collection of all locally finite connected graphs with a distinguished ordered pair of vertices up to isomorphism preserving this ordered pair. We often refer to an element of $\mathcalG$ as a {\it rooted graph} $(\G,o)$, without explicitly saying that we mean the equivalence class that it represents in $\mathcalG$. 
Let $(\G,o)$ be a random rooted graph and suppose that $o$ has finite expected degree.
Reweight the distribution of $(\G,o)$ by the Radon-Nikodym derivative ${\rm deg} (o)/\E({\rm deg} (o))$. We will refer to such a reweighting by saying that we {\it bias by the degree of the root}.
Denote the new random graph by $(\G',o')$. Let $X_0=o'$ and
let $X_1$ be a uniformly chosen neighbor of $X_0$. We say that $\G=(\G,o)$ is unimodular if 
$(\G',X_0,X_1)$ has the same distribution as $(\G',X_1,X_0)$. See \cite{BC} for the proof that this is equivalent to the original definition of unimodularity for graphs in \cite{AL}, which we recall in the next paragraph. 
One may consider some decoration or marking on rooted graphs, and extend the above definition in the obvious way. Whenever there is a decoration, given as a function $f$ on $V(\G)$ or as a subgraph $U\leq \G$, we denote this decorated rooted graph by $(\G,o;f)$, $(\G,o;U)$. In case of several decorations, we can list them all after the semicolon.

The original definition of unimodularity, equivalent to the previous one, is the following. Consider an arbitrary Borel function $f:\mathcalGG\to\R^+_0$. Then it has to satisfy the following equation
\begin{equation}\label{eqMTP}
\int\sum_{y\in V(G)} f(G,x,y)d\mu((G,x))=\int \sum_{y\in V(G)} f(G,y,x)d\mu((G,x)).
\end{equation}
Here we do not distinguish between $(G,x)$ as a rooted graph and as a representative of its equivalence class in $\mathcalGG$. This is standard in the literature and will not cause ambiguity.
Equation \eqref{eqMTP} is usually referred to as the ``Mass Transport Principle" (MTP). This equivalent definition of unimodularity naturally extends to decorated rooted graphs, one just has to consider Borel functions $f$ from the suitable space.


Let $M$ be some homogeneous metric space with some point $0$ fixed; for our purpose we can just assume that it is a Euclidean or hyperbolic space. Let $\Isom (M)$ be the group of isometries of $M$. 
For a graph $G$, 
an {\it embedding} $\iota$ of $G$ into a Euclidean or hyperbolic space $M$ is a map from $V(G)\cup E(G)$ that maps injectively every point in $V(G)$ to a point of $M$, and every edge $\{x,y\}$ to (the image of) a simple curve 
in $M$ between $\iota (x)$ and $\iota (y)$, in a way that two such images can intersect only in endpoints that they share. The embeddings that we consider are locally finite, hence $M\setminus \iota(V(G)\cup E(G))$ is open. The connected components of $M\setminus \iota (V(G)\cup E(G))$ are called {\it faces}. 

Let $(\G,o)$ be some unimodular random graph. For almost every $G$, let $\iota_G=\iota$ be some embedding of $G$ into $M$.
For every $v,w\in V(G)$, assign the label $\dist (\iota(v),\iota(w))$. (In the literature it is more standard to assign labels to the vertices or edges, but assigning them to pairs of vertices is also a possibility, which is essentially equivalent to the other notion; see \cite{BHK}.)
The embedded edges can also be encoded, as a label on each edge, coming from a suitable mark space in the space of continuous curves from $[0,1]$ to $M$ up to orientation-preserving isometries of $M$. 
From this labelling, one can reconstruct $\iota$ up to $\Isom (M)$. We say that $\iota$ is a {\it unimodular embedding} of $\G$ to $M$, if the labelling is a unimodular decoration of $\G$. We emphasize again that a $\iota$ and any $\iota'=\gamma\circ\iota$ ($\gamma\in \Isom (M)$) give rise to the same unimodular embedding.

A unimodular random graph $(\G,o)$ is {\it invariantly amenable} (or just {\it amenable}) if for every $\eps>0$ there is a random subset $U\subset V(\G)$ such that $(\G,o;U)$ is unimodular, every component of $\G\setminus U$ is finite, and $\Pr (o\in U)<\eps$. If this property fails to hold, then $\G$ is {\it invariantly nonamenable} (or just {\it nonamenable}). In the rest of this paper we will drop  ``invariantly'' and simply call unimodular random graphs amenable or nonamenable, but the reader should keep in kind that these terms are different from the ones used for a deterministic graph. For the relationship of this notion of amenability to almost sure amenability or anchored amenability, see the discussion in \cite{AL} after the definition, and Theorem 8.5 therein for some equivalents.


Consider a random graph $\calD$ drawn in $M$ in a measurable way, with a distribution that is invariant under $\Isom (M)$. Call such a $\calD$ invariant. The set $\calV\subset M$ of drawn vertices of $\calD$ forms an invariant point process in $M$. Say that the {\it intensity} of $\calD$ is the expected number of points of $\calV$ in a ball with unit volume. (This expectation does not depend on the location of the unit box, because of invariance.) Suppose that $\calD$ has finite intensity, and consider the {\it Palm version} $\calD ^*$ of $\calD$. By this we mean $\calD$ conditioned on $0\in \calV$. By standard theory of point processes and the assumption on finite intensity, this definition makes sense and $\calD^*$ is a random graph drawn in $M$ with a vertex in $0$.
Now let $(\G,o)$ be a unimodular random graph. By an {\it invariant embedding} (or {\it isometry-invariant embedding}) of $\G$ into $M$ we mean a random graph $\calD$ drawn in $M$ of isometry-invariant ditribution that has finite intensity, and with the property that $(\calD^*,0)$ viewed only as an element of $\mathcalG$ has the same distribution as $(\G,o)$. (The fact that $(\calD^*,0)$ is unimodular has been well-known, see Example 9.5 in \cite{AL}.)

\begin{remark}\label{Mecke}
While the term ``unimodular embedding'' seems to be used here for the first time, similar notions existed in the point process literature. A unimodular embedding tells the location of all embedded edges and vertices from the viewpoint of the root vertex. Specifically, the embedded vertices can be thought of as a point process with a point in the origin, and up to equivalence by isometries fixing the origin. The notion of {\it point-stationarity}, introduced by Thorisson \cite{Th} for processes with a point at the origin, requires an invariance under rerooting, similarly to the random walk definition for unimodularity. In fact, the definition of point-stationarity for point processes applies to embedded graphs right away, and the existence of a unimodular embedding is equivalent to the existence of a point-stationary graph whose underlying graph has the same distribution as the given graph. Our observation (from \cite{AL}) that the Palm version of an isometry-invariant embedding is unimodular has been essentially known since Mecke \cite{Me}, whose intrinsic characterization shows that the Palm version of a (translation-)invariant point process is always point-stationary.
\end{remark}

We will denote the Euclidean plane by $\R^2$, and the hyperbolic plane by $\H^2$. Note that an invariant embedding of a unimodular graph automatically has finite intensity, since it was required for the definition to make sense. 
We are interested in embeddings where in addition, no bounded open set is intersected by infinitely many embedded edges, that is, {\it locally finite embeddings}. The notion of local finiteness is invariant under isometries, so one can define it for unimodular embeddings as well.
By a {\it tiling} we mean a collection of pairwise disjoint connected polygons (``tiles'') such that the union of their closures is $M$. We will be interested in tilings where every compact subset of $M$ is intersected by only finitely many types. A tiling defines a graph where the tiles are the vertices and two of them are adjacent if they share some nontrivial line segment on their boundary.
The definition of invariant tilings that represent $(\G,o)$ is similar to that of invariant embeddings. Namely, suppose that a random tiling is isometry-invariant and the tile of the origin has finite area almost surely. Take a uniform random point in every tile and consider the Palm version of this point process together with the tiles on it. If the rooted graph defined by this tiling (rooted at the tile of the origin) has the same distribution as $(\G,o)$ then
we say that the invariant tiling represents $(\G,o)$.


\begin{remark}\label{infinitely_ended} 
It is well-known that an infinite unimodular random graph can have only 1, 2 or infinitely many ends, \cite{AL}. 
As mentioned earlier, Theorems \ref{dichotomy0} and \ref{dichotomy} do not cover the case of 2 ends and infinitely many ends.
If a unimodular random planar graph $\G$ has infinitely many ends almost surely, then it is nonamenable. There are examples where a unimodular embedding into $\H^2$ is possible, and examples when it is not (and similarly for invariant embeddings and for tilings). For the latter, let $G$ be the free product of the edge graph of a transitive hyperbolic tiling and a single edge, and $\G$ be supported in this single transitive graph. It is easy to check that any planar embedding of this graph is such that the embedded vertices have infinitely many accumulation points. Hence there is no locally finite invariant embedding or unimodular embedding for this graph. On the other hand, the 3-regular tree $T_3$ does have an invariant embedding into $\H^2$, which simply gives rise to a unimodular embedding. To see such an invariant embedding, take the Ford horocyclic tiling (see, e.g., Figure 3.3 in \cite{LP}), let $F$ be a fundamental domain that contains the origin, and choose a random isometry that maps the origin to a point of $F$ according to Haar measure. Consider
the centers of the interstices (bounded pieces in the complement of the disks) and
the straight lines between neighboring ones (which will all have the same length), and apply the random isometry to this embedded graph of $T_3$.
\end{remark}


The next example shows that ``zero intensity" is possible in case of {\it unimodular} embeddings. 
\begin{example}\label{example_strict}
Consider $\G$ to be $\Z$ almost surely, and let $M$ be the hyperbolic plane. Take an infinite geodesic $\gamma$ in $M$, fix a point $0\in \gamma$, and let $g$ be an isometry of $M$ that preserves $\gamma$ and maps $0$ to some $g(0)\not=0$. Consider the embedded graph with vertex set $\{g^i(0), i\in\Z\}$ and embedded edges being the pieces of $\gamma$ between pairs $g^i(0)$ and $g^{i+1}(0)$. One can check that this way we defined a unimodular embedding of $\Z$ into $M$. However, it is not possible to embed $\Z$ (or any amenable graph) into $M$ in an isometry-invariant way.
\end{example}
One could define intensity for unimodular embeddings. We will not need this, but it could be defined as the reciprocal of the unique number $r$ that ensures that the stable allocation on the embedded vertices with cell-volume-limit $r$ is a full allocation (with $r=\infty$ standing for zero intensity); see \cite{HP} for the definition. Forbiding zero intensity, one could rule out 
pathologies as Example \ref{example_strict}. Then Theorems \ref{dichotomy0} and \ref{dichotomy} would become full characterizations (by a suitable modification of the proofs in Section \ref{proofsofmain} for the added ``only if'' parts). 

A {\it cyclic permutation} of $n$ elements is a permutation that consists of a single cycle of length $n$.
A {\it combinatorial embedding} of a planar graph $G$ is a collection of cyclic permutations $\pi_v$ of the edges incident to $v$ over $v\in V(G)$, and such that 
there is an embedding of $G$ in the plane where the clockwise order of the edges on every vertex $v$ is $\pi_v$ for every $v\in V(G)$. 
A combinatorial embedding is unimodular if the decoration $\{\pi_v\}$ is a unimodular decoration. The notion of unimodular combinatorial embeddings (and maps) was implicitly introduced in Example 9.6 of \cite{AL}. Note that this definition does not use any underlying metric on the plane, as it defines an embedding only up to homeomorphisms. For a given edge $e$, choose an orientation of $e$ with $e^-$ being the tail and $e^+$ the head.
Consider the edge $\pi_{e^-} (e)$ oriented such that $e^-$ is its head. Repeat this procedure for this new oriented edge, and iterate until we arrive back to $(e^-, e^+)$. Call the resulting sequence of edges a {\it face} of
the combinatorial embedding. One can check that the faces of actual embeddings coincide with the bounded domains surrounded by the respective faces of the corresponding combinatorial embedding.
In \cite{Ti3} it is shown that being unimodular and planar guarantees the existence of a unimodular combinatorial embedding in general, see Theorem \ref{theorem3}.

\begin{remark}\label{implications}
The existence of a unimodular combinatorial embedding does not automatically provide us with a unimodular embedding into $\R^2$ or $\H^2$, but the other direction is obvious. To summarize: an isometry-invariant embedding defines a unimodular embedding (as verified in the proof of Theorem \ref{dichotomy0}), and a unimodular embedding trivially defines a unimodular combinatorial embedding. None of the other directions holds a priori.
\end{remark}

For a given graph $G$, a {\it circle packing} representation of $G$ is a collection of circles in the plane such that the circles are in bijection with $V(G)$ and two circles are tangent if and only if the corresponding vertices are adjacent in $G$. The {\it nerve} of a circle packing is the graph that it represents.
For a given circle packing $P$ in $\R^2$, consider
the union of all the disks in $P$ and their boundaries and its further union with all the bounded connected pieces (interstices) bounded by finitely many circles in $P$. Call the resulting set the {\it carrier} of $P$. 
In the first case they called the graph {\it CP parabolic}, while in the second case they called it {\it CP hyperbolic}. They found several  characterizing properties for this duality, such as the recurrence/transience of simple random walk. Earlier, Schramm \cite{S} proved the uniqueness of these circle packings up to some transformations.

\begin{theorem}{{\rm (He-Schramm, \cite{HS}, \cite{HS2}, Schramm, \cite{S}) }}\label{He_Schramm}
Let $G$ be a one-ended infinite planar triangulated graph. Then $G$ either has a circle packing representation 
whose carrier is the plane
or it has a circle packing representation 
whose carrier is the unit disk, but not both.
\begin{itemize}
\item In the former case (when $G$ is parabolic), the representation in the plane is unique up to isometries and dilations.
\item In the latter case (when $G$ is hyperbolic),  the representation in the unit disk is unique up to M\"obius transformations and reflections fixing the disk. 
\end{itemize} 
\end{theorem}
See also \cite{N}.

\subsection{Connections to past research}

Our topic is at the meeting point of isometry-invariant point processes and unimodular random graphs. The former has been a widely studied subject for many decades (see, e.g., the monographs \cite{Th}, \cite{LaPe}), mostly in the setup of stationary (translation-invariant) point processes in Euclidean spaces.  
A key problem in the present work is to represent certain graphs on the configuration points of a point process, in a covariant and measurable way. Questions of this flavor have been extensively studied for particular classes of graphs in the past, such as one-ended trees, biinfinite paths (\cite{FLT}, \cite{HP0}, \cite{Ti}) or perfect matchings (see, e.g., \cite{HPPS}, \cite{Ho}).
Note however, that while in these settings only graphs on the vertex set have to be defined, in our context we also need to {\it embed the edges} into the underlying space. 

Unimodular random graphs were first defined in \cite{AL}, but similar ideas existed earlier (see references in \cite{AL}).
They have attracted a lot of attention because of their connection to approximability by finite graphs (see \cite{P} for the importance of such approximability in group theory), because of closely related notions in other areas (such as graphings in measurable group theory, see, e.g., \cite{L}), and for being a natural generalization of group-invariant percolation.
One can think of the notion of unimodular random graphs as a generalization of ``percolation on a transitive graph with a unimodular group of automorphism, viewed from a fixed vertex'', which makes it analogous to the Palm version of a point process. The direct connection is that invariant point processes as well as unimodular random graphs satisfy the Mass Transport Principle. That the Palm version of a random graph invariantly drawn in the plane is always unimodular as a planar graph was already proved by Aldous and Lyons \cite{AL} (see also our Remark \ref{Mecke}), and our main results can be regarded as the converse to this claim.

A study of unimodular random planar graphs was initiated by Angel, Hutchcroft, Nachmias and Ray in \cite{AHNR} for the class of triangulations,
and they showed that for a locally finite ergodic unimodular triangulated planar simple graph, 
being CP parabolic is equivalent to invariant amenability. 
In \cite{AHNR2}, unimodular planar graphs were further studied, without the assumption of being triangulated, but with the assumption that the unimodular graph comes together with a unimodular combinatorial embedding, in which case this joint object is called a {\it unimodular planar map}. Several criteria were identified as equivalents to invariant amenability. 
Theorems \ref{dichotomy0} and \ref{dichotomy} can be thought of as further examples of the dichotomy.

As an example, consider the Uniform Infinite Planar Triangulation (UIPT), first defined by Angel and Schramm in \cite{AS}. The UIPT is a random graph that is unimodular (because it arises as the local limit of finite graphs) and planar (because all these finite graphs are planar), and moreover, the graph comes together with a unimodular combinatorial embedding (inherited from the finite graphs). Nevertheless, a combinatorial embedding defines an actual embedding into the plane only up to homeomorphisms. 
Now, if we are given not only the topology but also some homogeneous metric on the plane (and the the only homogeneous simply connected Riemannian manifolds of infinite volume, up to scaling, are the Euclidean and the hyperbolic planes, see Theorem 3.8.2 in \cite{Thu}),
then it is natural to require the unimodular embedding to exist not just up to homeomorphisms but up to isometries. This latter, a unimodular embedding up to isometries, is what we defined simply as a unimodular embedding. If a unimodular embedding exists, one can further ask whether an isometry-invariant embedding exists with this given unimodular embedding. This last question takes us back to the theory of point processes, where analogous questions were addressed: is a unimodular point process (that is, a point process with a point in the origin and satisfying a MTP) always the Palm version of a isometry/translation-invariant point process? This question was settled by Heveling and Last \cite{HL} in the positive (who use the term point-shift stationary for what we called here a unimodular point process) for translation invariant processes. 
 
Here we are interested in planar graphs, but it is reasonable to ask what happens in higher dimensions. In \cite{Ti2} it is shown (in the dual language of tilings) that every one-ended amenable unimodular transitive graph has an isometry-invariant embedding into $\R^d$ when $d\geq 3$. The proof generalizes from transitive to random unimodular graphs. To our knowledge, the nonamenable (hyperbolic) case is open; Question 3.6 of \cite{Ti2} may be relevant in this regard.

\section{Unimodular planar triangulation of unimodular planar graphs}

The following theorem was proved in \cite{Ti3}.

\begin{theorem}{{\rm (\cite{Ti3})} }\label{theorem3}
Let $(\G,o)$ be a unimodular random planar graph of finite expected degree. Then $(\G,o)$ has a unimodular combinatorial embedding into the plane.
\end{theorem}
Recall that having a unimodular combinatorial embedding is a weaker requirement than having a unimodular or an invariant embedding, see Remark \ref{implications}.

\begin{theorem}\label{triangulate}
Let $(\G,o)$ be a unimodular planar graph of finite expected degree. Then there is a unimodular decorated graph $(\G^+,o^+, S)$ where $S$ is a connected subgraph of $\G^+$ and such that $(S,o^+)$ conditioned on $o^+\in S$ has the same distribution as $(\G,o)$. 
and $\G^+$ is a planar triangulation of finite expected degree. If $\G$ has one end then $\G^+$ has one end.
\end{theorem}

\begin{proof}
By Theorem \ref{theorem3}, $\G$ has a unimodular combinatorial embedding into the plane. Fix such an embedding. The collection of faces is also jointly unimodular with $\G$.

Let $F$ be an unbounded face. By exploring the vertices of $F$ along the boundary, a function from $\Z$ to the boundary is obtained which is not necessarily injective. Fix such a bijection, choose $\xi\in\{0,1\}$ uniformly at random, and for every pair $\{2k+\xi, 2k+\xi+1\}$ ($k\in\Z$), add a new vertex $v_k$ to the graph, and connect it to $2k+\xi$, $2k+\xi-1$ and $2k+\xi+1$. Finally, add an edge between $v_k$ and $v_{k+1}$ for every $k\in\Z$. Now, in the resulting new graph we have a new infinite face, whose boundary is the biinfinite path induced by $\ldots,v_{-1},v_0,v_1,\ldots$. Repeat the previous procedure for this biinfinite path, and so on, ad infinitum. 

For every bounded face $F$ do the following. If $F$ has $n$ edges on its boundary, then add a new cycle $C$ of length $[(n+1)/2]$ inside this face. Add edges that connect these new vertices to the boundary vertices of $F$ so that planarity is not violated, in such a way that we connect every vertex of $C$ to the vertices of two consecutive edges of $F$, except maybe one vertex of $C$ which is connected to one edge of $F$.
Then for every vertex $x$ of $F$ and two consecutive edges of $F$ containing $x$, there are at most two new edges added to $x$ between these two consecutive edges. Hence the degree of $x$ is at most trippled by the end of this procedure.
Repeat this step for the new face, surrounded by $C$, as long as $|C|\geq 6$. Otherwise just add a new vertex and connect it to every vertex of $C$. In each step, if there are more than one options, choose one of the options randomly and uniformly.


When doing this for every $F$, in the limit we get a triangulation $\G^+$.
All the operations preserved planarity, hence $\G^+$ is planar.
One can define a function $\tau : V(\G^+)\to V(\G)$ (possibly using extra randomness) such that $\E |\tau^{-1} (o)|<\infty$. Therefore, by similar arguments to Example 9.8 of \cite{AL}
(see Subsection 1.4 in \cite{BPT} for more details), $\G^+$ is unimodular with the random root $o^+$ chosen as follows: we first bias $(\G,o)$ by weights proportional to $|\tau^{-1} (o)|$, and then $o^+$ can be sampled by choosing a uniform element of $\tau^{-1} (o)$.

From the same argument in \cite{BPT} it follows that the expected degree of $o^+$ is bounded by $\max \{3D,7\}$, where $D$ is an upper bound on the expected degree of the root in $(\G,o)$.  Here we are using the fact that the vertices in $V(\G)$ have received at most twice as many new edges as their degree in $\G$ (at most two new edges between each pair of consecutive edges on a face), and every new vertex has at most 7 incident edges.
\end{proof}

Before proceeding, let us state some consequences of
Theorem \ref{theorem3} related to percolation, as the methods of 
\cite{BS} become available for the planar graph that have a unimodular combinatorial embedding. This corollary was unnoticed in \cite{Ti3}, the source of Theorem \ref{theorem3}, so we present it here.
Call a planar graph $G$ {\it edge-maximal}, if for any $x,y\in V(G)$, such that $\{x,y\}\not\in E(G)$ and $x\not=y$, $G\cup \{\{x,y\}\}$ is nonplanar. Edge maximality means that any planar embedding of $G$ is triangulated.

\begin{corollary}\label{percolation}
Let $\G$ be an ergodic nonamenable unimodular random planar graph with finite expected degree and one end. 
Then there exist $p_c$ and $p_u>p_c$ such that Bernoulli($p$) percolation has no infinite component for $p\in [0,p_c]$, it has infinitely many infinite components for $p\in (p_c,p_u)$, and it has a unique infinite component for $p\in [p_u,1]$.
If $\G$ is edge-maximal and we percolate on the vertices, then $p_c<1/2$.
\end{corollary}

\begin{proof}
The claim that there is no infinite component in $p_c$ is known to be true for any nonamenable graph, see Theorem 8.11 in \cite{AL}. That $p_c<p_u$ and there is uniqueness in $p_u$ are proved for graphs where a unimodular planar combinatorial embedding is given in Theorem 8.12 in \cite{AL} (based on \cite{BS}); combined with Theorem \ref{theorem3} (a result from \cite{Ti3}) they show the claims here. Finally, for the claim about $p_c<1/2$, the proof of Theorem 6.2 in \cite{BS} extends, once we have a unimodular combinatorial embedding as in Theorem \ref{theorem3}. 
\end{proof}

A bound $p_c<1/2$ is proved in \cite{HaPa} for general graphs with a minimum degree requirement.

\section{Invariant circle packing representations of nonamenable graphs}


The next theorem is a simple consequence of results by He and Schramm.


\begin{theorem}\label{hyper_embed}
Suppose that $\G=(\G,o)$ is a one-ended nonamenable ergodic unimodular random planar simple graph with finite expected degree. Then $\G$ can be represented by a unimodular circle packing in the hyperbolic plane.
Consequently, $\G$ has a unimodular embedding into the hyperbolic plane $\H$, and $\G$ can be represented by an invariant tiling. 
\end{theorem}


\begin{proof}[Proof of Theorem \ref{hyper_embed}]
We may assume that $\G$ is triangulated. Otherwise apply first the proof for the triangulated supergraph
$\G^+$ in Theorem \ref{triangulate}, and then only keep the circles that represent vertices in $V(\G)$.
We know from He and Schramm \cite{HS} that $\G$ has a unique representation in the hyperbolic plane up to hyperbolic isometries.
From the proof by He and Schramm it follows, as explained in detail in Subsection 3.4.1 of \cite{AHNR}, that 
the hyperbolic radius of the circle representing vertex $v$ is a measurable function of $(G,v)\in\mathcalG$.
Hence the circle packing ($\G$ as a graph marked with the unique circle packing, where the marks are telling the embedding of the vertices together with an extra label representing the radius) is unimodular.

One can turn the circle packing into a tiling of the same adjacency structure by properly subdividing every component of the complement of the disks into finite pieces and attaching them to suitably chosen neighboring disks. We omit the details. 
\end{proof}

\section{Invariant embeddings of amenable unimodular planar graphs}

In this section we will prove the amenable part of Theorem \ref{uj}. We will later use this invariant embedding to constuct a unimodular embedding and prove the amenable parts of Theorems \ref{dichotomy0} and \ref{dichotomy}. One may wonder if a unimodular embedding could be found directly, 
following the lines of the proof for the hyperbolic case, but a key part which does not go through is the following. The uniqueness in the He-Schramm Theorem \ref{He_Schramm} is up to isometries in the nonamenable (hyperbolic) case, and hence it could be used in the construction of the unimodular embedding, which is also defined only up to isometries. Now, in the amenable case, the uniqueness in Theorem \ref{He_Schramm} is only up to isometries {\it and dilations}, which makes the above method fail. 

\begin{quest}
Can every amenable unimodular planar graph be represented by an invariant circle packing in $\R^2$?
\end{quest}
At the time of submission of this manuscript, Ali Khezeli found a negative answer to the question, \cite{Kh}. His counterexample satisfies the stronger property that it has no {\it unimodular} circle packing representation.
The case of UIPT is still open. See also Question \ref{diameter_question}.

Let $\G$ be a unimodular random graph. Let ${\mathcal G}_1, {\mathcal G}_2,\ldots$ be a (random) sequence of partitions of $V(\G)$ such that the collection $(G,{\mathcal G}_1, {\mathcal G}_2,\ldots)$ is unimodular. Say that such a partition sequence 
is a {\it unimodular finite exhaustion}, if it is coarser and coarser, if every class ({\it part}) of every partition is finite, and any two vertices are in the same part of ${\mathcal G}_i$ if $i$ is large enough. 

The definition of amenability for locally finite unimodular random graphs is equivalent to the existence of a unimodular finite exhaustion (see Theorem 8.5 in \cite{AL} and references therein). To see this, let $U_k$ be the random subsets corresponding to $\eps=2^{-k}$ in the definition of amenability in Section \ref{intro}, and let ${\mathcal G}_n$ consist of the connected components of $\G\setminus \{e\in E(\G), e\cap \cup_{k=n}^\infty U_k\not=\emptyset\}$.

\begin{theorem}\label{parab_embed}
Suppose that $(\G,o)$ is a one-ended amenable unimodular planar random graph of finite expected degree.
Then $\G$ has an invariant locally finite embedding into $\R ^2$
such that the image of every edge is a broken line segment. 
\end{theorem}

Some well studied planar unimodular amenable graphs are the {\em uniform infinite planar triangulation} (UIPT) (\cite{AS})
and the augmented critical Galton-Watson tree conditioned to survive (AGW) (see, e.g., \cite{LP}).
 

\begin{corollary}\label{UIPT}
The AGW and the UIPT can be invariantly embedded in the Euclidean space.
\end{corollary}

We will need a special case of Theorem 5.1 from \cite{Ti2}, illustrated on Figure \ref{regibol}.

\begin{theorem}{{\rm (\cite{Ti2}) }}\label{fafaktorbol2}
Consider an ergodic one-ended unimodular random tree $(T,o)$ of uniformly bounded degrees. Then there is an invariant locally finite tiling in $\R^2$ that represents $T$.
\end{theorem}

\begin{figure}[h]
\vspace{0.1in}
\begin{center}
\includegraphics[keepaspectratio,scale=0.5]{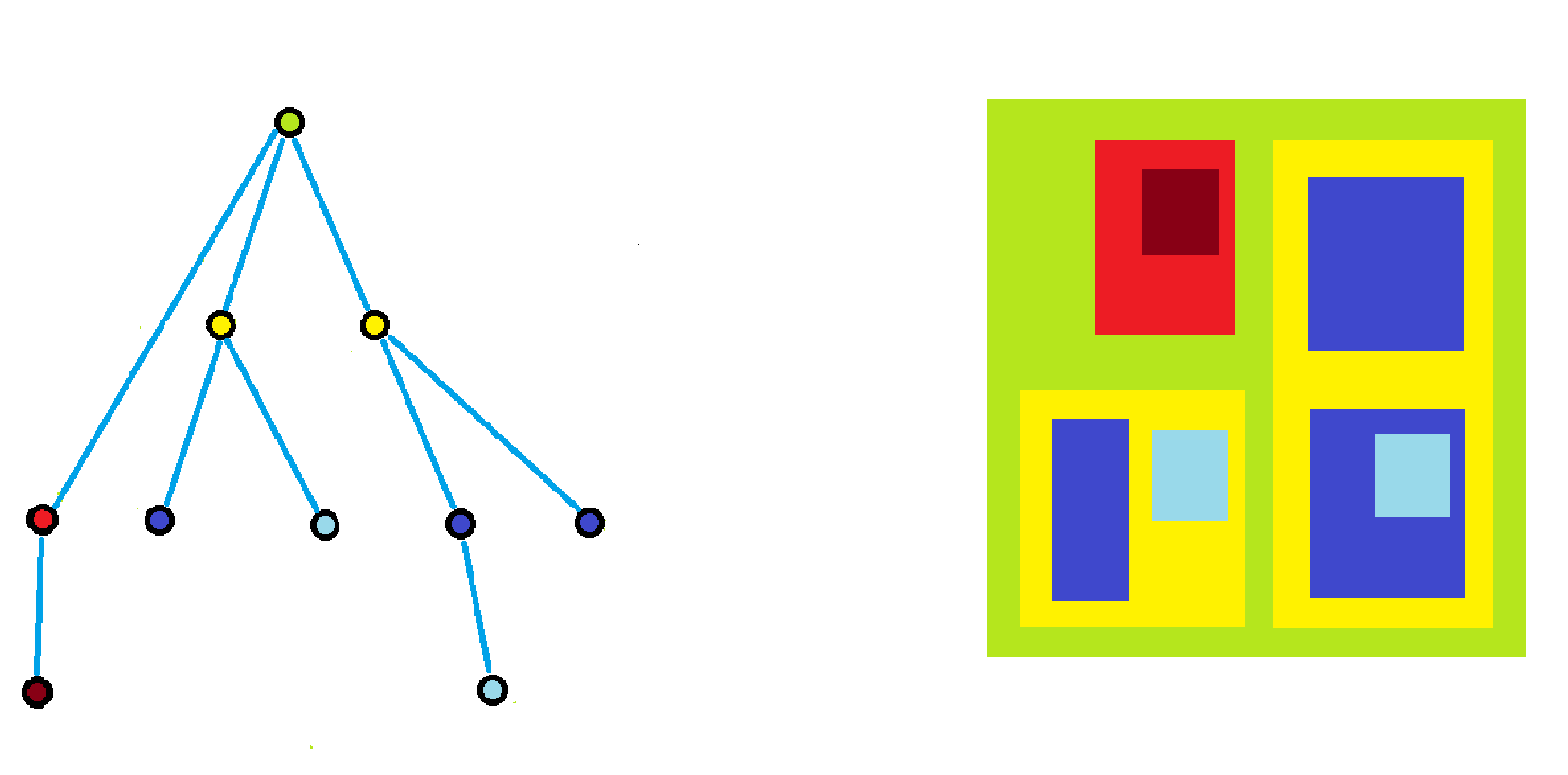}
\caption{Representing a tree by a tiling; image taken from \cite{Ti2}.}\label{regibol}
\end{center}
\end{figure}

\begin{proof}[Proof of Theorem \ref{parab_embed}]
In the following constructions it will often happen that we need to choose some collection of pairwise inner-disjoint curves (broken line segments) connecting some given collection of pairs of points, within some specified bounded domain. We will want to do it so that for a given isometry-invariant and measurable random collection of pairwise disjoint domains in $\R^2$ the resulting collection of line segments over all domains is also invariant and measurable.
Let us describe a method to do so. For each of the specified bounded domains do the following. Fix a random origin in the domain and two uniformly chosen axes for $\R^2$. Then
make the choice for the broken line segments so that the pairwise inner-disjointness is satisfied, and every breaking point of every segment has dyadic coordinates of the form $h2^{-k}$, where $h\in\Z$ and $k\in \Z^+$ is minimal such that the choice of all broken line segments with the above constraints is possible. If there is more than one such choice with this minimal $k$, then choose one of them uniformly at random. So whenever we are making such choices of families of curves, embedded edges etc. in the future, this is how we understand it, without further mention.

Fix some unimodular combinatorial embedding $\Sigma=\{\sigma_v:\, v\in V(\G)\}$
of $\G$ (as given by Theorem \ref{theorem3}), where $\sigma_v$ is the permutation on the edges incident to $v\in V(\G)$. Suppose that an embedding of some subgraph $H$ of $\G$ is given, together with an embedding of all the half-edges from $E(\G)\setminus E(H)$. (By an embedding of a half-edge we simply mean the drawing of a broken line segment starting from the vertex, together with a label that tells the other endpoint of the corresponding edge.) We say that this embedding of the edges and half-edges is {\it consistent with} $\Sigma$, if the permutation determined by the embedded (half-)edges around $v$ in the positive direction is $\sigma_v$.

We will prove a stronger claim than the theorem, namely, that there exists an invariant embedding into $\R^2$ that is consistent with 
$\Sigma$.

We may assume that $\G$ has uniformly bounded degree, as we explain next. We will introduce a new graph $\G'$, that we will obtain from $\G$ by replacing vertices by paths, with two such path adjacent if and only if the original vertices were, and in a way that all degrees in $\G'$ will be at most 3. The construction will give rise to a combinatorial embedding $\Sigma'$ of $\G'$.
Then we will show that an invariant embedding of $\G'$ consistent with $\Sigma'$ can be used to define one for $\G$ that is consistent with $\Sigma$. The simple construction is summarized on Figure \ref{bounded_deg}.  So define $\G'$ as follows.
For every vertex $v\in V(\G)$ let $n(v,1),\ldots, n(v,k)$ be the listing of its neighbors given in the order by $\sigma_v$, with the starting neighbor $n(v,1)$ chosen at random, where $k=k(v)$ is the degree of $v$. The vertex set $V(\G')$ will be the union of $v(1),\ldots, v(k)$ over all the $v$. Now define edges of $\G'$ as the collection of all pairs $\{v(i),v(i+1)\}$ (as $v\in V(\G)$, $i=1,\ldots, k(v)-1$), and pairs $\{v(i), w(j)\}$ (as $v,w\in V(\G)$, $n(v,i)=w$ and $n(w,j)=v$). By first biasing $(\G,o)$ by the degree of the root and then replacing its vertices $v$ by the $v(i)$ (choosing the new root uniformly from the vertices replacing the original one) with edges as in $\G'$, we obtain a rooted graph $(\G',o')$, which is unimodular (as shown in Example 9.8 of \cite{AL} or Subsection 1.4 in \cite{BPT}). 
Finally define the unimodular combinatorial embedding $\Sigma'$ of $(\G',o')$: let $\sigma'_{v(i)}$ be the cyclic permutation $\bigl(v(i-1), w(j), v(i+1)\bigr)$ on the neighbors of $v(i)$ whenever $1<i<k(v)$, and (trivially) $\bigl(v(2),w(1)\bigr)$ for $i=1$ and $\bigl(v(k-1),w(k)\bigr)$ for $i=k$.
Suppose we find an invariant embedding consistent with $\Sigma'$ for the bounded-degree graph $\G'$; let $\phi (x)$ be the location of a vertex $x$ by this embedding, and let $L(v(i), w(j))=L(w(j),v(i))$ be the drawn edge between $v(i)$ and $w(j)$ when $v,w\in V(\G)$ and
$\{v(i),w(j)\}\in E(\G')$. (Everything that we are doing in the rest of the paragraph is covariant, hence we can fix a representative in the Palm version of the embedded graph and hence refer to actual embedded points.)
For every $v\in V(\G)$ pick an $\iota_v\in \{1,\ldots, k(v)\}$ at random. The embedded subtree induced by $v(1),\ldots, v(k)$ can be used to define broken line segments $L(v(i), v(\iota_v))$ between $\phi (v(i))$ and $\phi (v({\iota_v}))$, for all $i\not=\iota_v$, such that all these broken line segments over the various $v\in V(\G)$ are pairwise inner-disjoint, and also inner-disjoint from all the $L(v(i),w(j))$,
$\{v(i), w(j)\}\in E(\G')$, $v\not=w$.
To avoid unnecessary formalism we leave the further details to the interested reader; we only mention that an $\varepsilon$-close contour walk around the embedded subtree of edges incident to $v(1),\ldots, v(k)$ can give us guidance about where to draw the edges so that the above constraints are satisfied. Define $L(v(\iota_v), v(\iota_v))=\emptyset$.
Now, to draw an edge $\{v,w\}\in E(\G)$, consider the union of $L(v(i), v(\iota_v))$, $L(w(j), w(\iota_w))$ and $L(v(i), w(j))$, where $n(v,i)=w$ and $n(w,j)=v$. 
We end up with an embedding of $\G$ from the embedding of $\G'$ as desired; see Figure \ref{bounded_deg}. So from now on we will assume that $\G$ has bounded degrees.

\begin{figure}[h]
\vspace{0.1in}
\begin{center}
\includegraphics[keepaspectratio,scale=0.8]{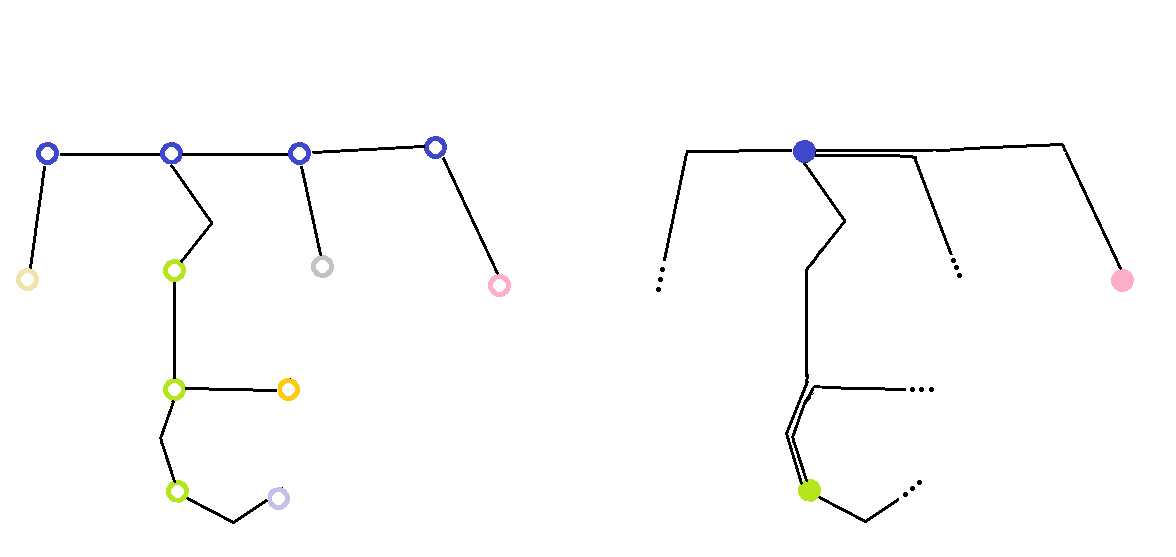}
\caption{The scheme of getting an embedding of $\G$ from the embedding of its bounded degree version $\G'$. Dots of the same color stand for the $v(1),\ldots, v(k(v))$, the solid dot indicates $v({\iota_v})$.}\label{bounded_deg}
\end{center}
\end{figure}

Now choose a one-ended unimodular spanning tree $T$ of $\G$ and apply Theorem \ref{fafaktorbol2} to get a tiling representation of it. If $\tau (v)$ is the tile representing $v\in V(T)$, then we choose a random point $\iota(v)$ of $\tau(v)$ to be the embedded image of $v$. 

For an $x\in T$ let $T_x$ be the finite subtree of $T$ induced by the union of all finite components of $T\setminus\{x\}$ and $\{x\}$. Say that the depth of $T_x$ is $k$ if the largest distance from $x$ within $T_x$ is $k$, and define 
$$L_k:=\{x\in V(T): T_x \text{ has depth } k\}.$$ 
If $x,y\in L_k$, $x\not=y$, then $T_x$ and $T_y$ are disjoint. 
Let $G^i$ be the graph on vertex set $V(G)$ and edge set
$$E(G^i):=\bigl\{\{x,y\}:\, x,y\in T_v \text{ for some } v\in L_i \bigr\}.$$
Then $G^i$ is a unimodular finite exhaustion of $G$.
Finally, define $\gamma (v)=\tau (v)$ whenever $v\in L_0$, and let
$\gamma(v)$ be the ball of radius $\dist (v,\partial\tau(v))/2$ around $v$ when $v\not\in L_0$. Let $\Gamma(v)$ be the interior of the closure of $\cup_{u\in T_v} \tau (u)$. 

The component $K_v$ of every $v\in L_i$ in $G^i$ is such that $G\setminus K_v$ is connected. Furthermore, $\Gamma(v)$ is simply connected, and $\iota (u)\in \Gamma(v)$ for every $u\in V(K_v)=V(T_v)$. 
We will embed the edges of $K_v$ in $\Gamma (v)$, and the above facts guarantee that no two points of $\iota (V(G\setminus K_v))$ are separated in $\R^2$ by these embedded edges. Furthermore, for all edges not in $K_v$ but incident to $K_v$, we will define an embedded half-edge, connecting the endpoint to $\partial\Gamma(v)$. The procedure is illustrated on Figure \ref{novesztes}.

\begin{figure}[htbp]
\begin{center}
\includegraphics[keepaspectratio,scale=0.8]{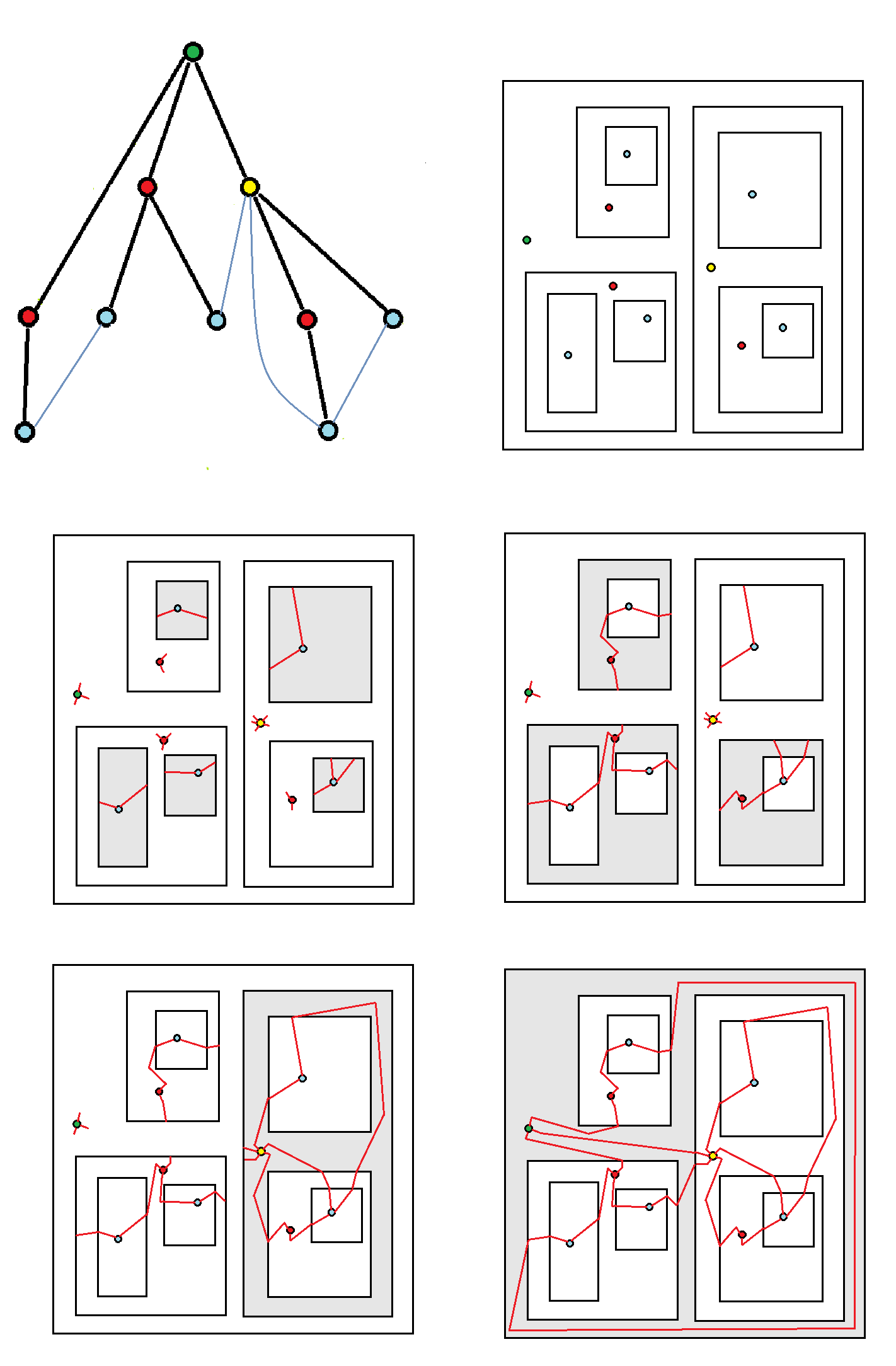}
\caption{A subtree $T_v$ of $T$ is taken (as the tree of Figure \ref{regibol}, with vertices recolored), together with the other edges of $G$ induced by it, producing the component $K_v$ of $G_n$. Blue vertices are in $L_0$, red ones in $L_1$, yellow is for $L_2$, green for $L_3$.
The embedding of edges is shown over several steps. Shaded areas indicate the $\tau(w)$, $w\in L_i$, where edges are drawn in step $i$. They are not used in any later step.}\label{novesztes}
\end{center}
\end{figure}

We will proceed in steps, embedding $G^n$ in step $n$, and also embedding the half-edges coming from $G\setminus G^n$.
This will be done in a way that 
\begin{enumerate}
\item the embedding of $G_n$ is an extension of that of $G_{n-1}$. In particular, for every $e\in G_n\setminus G_{n-1}$ the embedded image of $e$ contains its embedded half-edges from the previous step, and for every edge of $G_{n-1}$ the embedding in step $n$ is the same as in step $n-1$.
\item All edges $\{v,w\}$ with $w\in K_v$, $v\in L_n$, are embedded in the interior of the closure of $\tau(v)\cup\Gamma (w)$;
\item similarly, if an edge in $ E(K_v)\setminus G^{n-1}$ has its endpoints in the distinct subtrees $T_w$ and $T_{w'}$ of $T_v$, with $w$ and $w'$ adjacent to $v\in L_n$, then the edge is embedded into the interior of the closure of $\tau (v)\cup\Gamma (w)\cup\Gamma (w')$.
\item All half-edges starting from a vertex in $K_v$, $v\in L_n$, reach $\partial\Gamma(v)$.
\end{enumerate}
As a preparatory step (step 0), consider $G^0$, the empty graph on $V(G)$. For every $v\in V(G)$, pick some embedding of the half-edges starting from $\iota(v)\in\gamma(v)$ to randomly chosen points of the boundary of $\gamma(v)$, in a way that the embedding of the half-edges is consistent with $\Sigma$.
The embedding defined for step 0 trivially satisfies (1)-(4). It is consistent with $\Sigma$ by definition, and any extension will also be consistent with $\Sigma$.
We proceed to step $n$ recursively. For every component $K_v$ of $G^n$ ($v\in L_n$)
do the following. Let $K_{v_1},\ldots, K_{v_k}$ be the components of $G^{n-1}$ induced by $V(K_v)$ ($v_i\in\cup_{j=0}^{n-1}L_j$). 
Then $\Gamma(v)$ contains the embedded points of $V(K)$, moreover, the
$\Gamma(v_i)$ are pairwise disjoint open subsets of $\Gamma(v)$ containing the embedded $K_{v_i}$ respectively.
Contract every $K_{v_i}$ in $\G$ to a single point, and also contract $\G\setminus K_v$ (which is a connected graph) to a single point $x_\infty$. The resulting finite graph $K'$ inherits a combinatorial embedding from $\Sigma$. Consider an arbitrary embedding of $K'$ to the 2-sphere that represents this combinatorial embedding. Remove an infinitesimally small neigborhood of the embedded nodes, to get a 2-sphere with holes in it and pairwise disjoint arcs connecting some of these holes. It is easy to check that there is a homeomorhism from this surface to $\tau(v)$, which maps the boundary of the hole belonging to $x_\infty$ to $\partial \Gamma (v)\subset\partial\tau (v)$, and maps the boundary of the hole corresponding to the contracted $K_{v_i}$ to $\partial \Gamma(v_i)\subset\partial \tau(v)$. Furthermore, this homeomorphism can be chosen so that the images of the drawn segments are continuations of the respective drawn half-edges in the $\Gamma (v_i)$. We have just shown that there exists an embedding of the component $K_v$ of $G_n$ into $\Gamma(v)$ that satisfies (1)-(4); now choose the broken line segments for one such embedding randomly in a way as described at the beginning of this proof.

Since every step is an extension of the previous one, we obtain an embedding of $G$ in the limit, as desired. Local finiteness is guaranteed by the fact that every $\tau (u)$ ($u\in V(G)$) is intersected by finitely many embedded edges.
\end{proof}


\begin{theorem}\label{parab_tiling}
Every amenable unimodular random planar graph $(\G,o)$ can be represented in $\R^2$ as the neighborhood graph of an invariant random tiling.
\end{theorem}

In \cite{Ti2} it is proved that every amenable unimodular transitive graph can be represented by an invariant tiling of $\R^d$ for $d\geq 3$, and the proof extends from transitive to random right away.

\begin{proof}
Consider the embedding of $\G$ into $\R^2$ as in Theorem \ref{parab_embed} and let $\Pp$ be the point process that the embedded vertices define. To each face $F$ and point $v\in \partial F$, $v\in \Pp$, we will assign a piece of the face incident to $v$, in such a way that two such pieces share a 1-dimensional boundary iff the corresponding vertices are adjacent. For the case of bounded faces one can apply a modified ``barycentric subdivision", see Figure \ref{baricentric}: for each pair $v$ and $w$ of adjacent vertices that are consecutive along $F$, consider the broken line segment representing the edge between them, and consider its midpoint, that is, the point that halves the length of the broken line. Choose some point uniformly in $F$, and connect it to all these midpoints by some broken line.
If $F$ is infinite, we will apply a trick similar to the one in \cite{K}. For every pair $v$ and $w$ of adjacent vertices such that $\iota (v)$ and $\iota (w)$ are consecutive along $F$, let $h(v,w)=h(w,v)$ be the midpoint of the broken line segment between them. Choose a conformal map $f$ between $F$ and the upper half plane $\H$ of $\C$ that maps infinity to infinity. By the standard extension of $f^{-1}$ to the boundary $\partial \H$,  we can define a set of $f$-images in $\R$ for every $h(v,w)\in\partial F$. (This set consists of one or two points, depending on whether the broken line between $v$ and $w$ has $F$ on only one side or on both.) Let $a\in R$ be one such image, and consider the vertical line $L_a=\{a+bi\,:\,b\in\R^+\}$ in $\H$. Consider $f^{-1}(L_a)$ for all the $a$. One can check that they subdivide $F$ into pieces as we wanted.
It is also clear that the construction does not depend on the choice of $f$ (which is unique up to conformal automorphisms of the upper half plane of the form $x\mapsto ax+b$, $a,b\in\R$, $a\not=0$), and that it is invariant. See \cite{K} for a detailed argument.
\begin{figure}[h]
\vspace{0.1in}
\begin{center}
\includegraphics[keepaspectratio,scale=1.4]{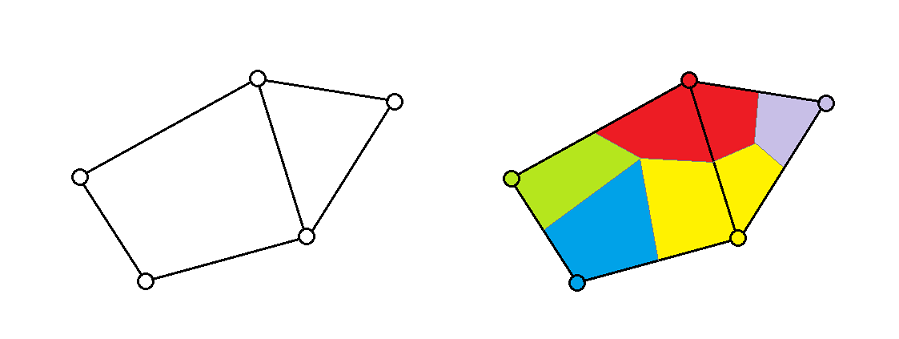}
\caption{Splitting up a face to subtiles. Broken line segments are represented by straight segments for simplicity.}\label{baricentric}
\end{center}
\end{figure}

It remains to prove that the tile of the origin has finite area almost surely. This is known for any invariant point process in $\R^2$ (whose intensity is automatically positive) and partition as in our setup, see, e.g., (9.15) in \cite{CSKM}.
\end{proof}

We expect that with some extra work one could also get a tiling where every tile has area 1. One would have to ensure that the embedded vertices in Theorem \ref{parab_embed} form a point process where the number of points in large boxes is relatively close to the expectation, and then build up the tiling stepwise and directly, eventually assigning unit tiles to all vertices. We have not worked out the details.

What seems to be harder to control, is the diameter of the tiles.

\begin{quest}\label{diameter_question}
What can we say about the distribution of the diameter of a tile in a construction as Theorem \ref{parab_tiling}? How fast can it decay?
\end{quest}
Various invariance principles follow from 
\cite{GMS} 
if one is able to construct an initial embedding for the given graph that satisfies a certain finite energy condition. The embeddings are assumed to be translation invariant modulo scaling. Whether our method can be useful in this setting is to be investigated in the future.

\section{Proofs of the main theorems}\label{proofsofmain}

\begin{proof}[Proof of Theorem \ref{uj}]
The existence of such representations if $\G$ is amenable is
proved in Theorems \ref{parab_embed} and \ref{parab_tiling}.

For the ``only if" part, suppose first that a nonamenable $\G$ had an isometry-invariant embedding as in Theorem \ref{dichotomy0} into $\R^2$. Then one could use the invariant random partitions of $\R^2$ to $2^n$ times $2^n$ squares to define a unimodular finite exhaustion of $\G$. Thus $\G$ has to be amenable, a contradiction.
\end{proof}

\begin{proof}[Proof of Theorem \ref{dichotomy0}]
For the nonamenable case the unimodular embedding into $\H^2$ is given in Theorem \ref{hyper_embed}. That there is no such an embedding into $\R^2$ is proved the same way as in the proof of Theorem \ref{uj}.

If $\G$ is amenable, an isometry-invariant embedding into $\R^2$ exists. With the same arguments as in Example 9.5 of Aldous and Lyons in \cite{AL}, 
the Palm version of this random embedded graph (as a graph with the decoration given by the embedding and rooted in the origin) is unimodular.
\end{proof}

\begin{proof}[Proof of Theorem \ref{dichotomy}]
The ``if" parts of the claims follow from Theorem \ref{parab_tiling} and \ref{hyper_embed}. (Although Theorem \ref{hyper_embed} is for simple graphs, the tiling obtained from a circle-packing can be extended when there are parallel edges or loops.)

For the ``only if" part, note that an invariant tiling gives rise to an invariant embedding (choose a uniform random point in each tile and suitably connect it to its neighbors). Hence the claim is reduced to that in Theorem \ref{dichotomy0}.
\end{proof}

\section{Further directions and open problems}\label{concluding}
To conclude, we propose a number of questions that are more or less connected to our main topic.

A natural direction could be the following. Say that two unimodular Lie groups $M$ and $N$ are equivalent, if every, possibly random,
tiling $T$ that invariantly tiles $M$ with compact tiles also invariantly tiles $N$.
Here tiles are compact simply connected.
Assume there is a tiling that invariantly tiles both $M$ and $N$. 
\begin{itemize}
\item Must they be quasi-isometric? We found that this is not necessarily the case. As proved in \cite{Ti2}, $R^d$ can be invariantly tiled with bounded tiles by any amenable transitive one-ended graph, whenever $d\geq 3$. In particular, both $\R^3$ and $\R^4$ can be invariantly tiled by $\Z^3$, yet they are not quasi-isomorphic.
\item Must they be equivalent? Not necessarily. $\Z^2$ invariantly tiles $\R^2$, and also $\R^3$ (by the just mentioned result). But the two are not equivalent, because no nonplanar graph (such as $\Z^3$) can tile $\R^2$.
\end{itemize}


Bonk and Schramm \cite{BonkS} constructed a quasi-isometric embedding of hyperbolic graphs into real hyperbolic spaces.

\begin{quest}
Is there an invariant quasi-isometric embedding of unimodular hyperbolic graphs into a real hyperbolic space $\H^d$? Or is there such a unimodular embedding?
\end{quest}

Consider some infinite graph, and partition it to infinitely many (roughly) connected infinite subgraphs, such that each part neighbors only finitely many other parts.
Which Cayley graphs admit {\it invariant random partitions} (IRP)?
(Variants of this question can further require that the parts are indistinguishable or removing the ``finite number of neighbors" requirement.)
Together with Damien Gaboriau and Romain Tessera we observe that
a Cayley graph of a group with positive  first $L^2$ Betti number does not admit an IRP.
As an exercise, show that the regular trees do not admit IRP.

With Romain Tessera we conjecture that the lamplighter over $\Z$ does not admit an IRP.
What about $SL_3(\Z)$?

Given an invariant random partition, when is it possible to further partition each part to infinitely many (roughly) connected infinite subgraphs?
Think of the examples $\Z^2$, $\Z^3$ and $ T \times \Z$.
The number of possible iterations might be of interest.

\medskip

\noindent
{\bf Acknowledgments:} We are indebted to an anonymous referee for many valuable improvements on the paper, as well as for finding an error and suggesting some ideas for the correction.
Thanks to G\'abor Pete, Omer Angel, Sebastien Martineau, Romain Tessera and Ron Peled for useful discussions, and to Andr\'as Stipsicz for a reference. The second author was supported by the ERC Consolidator Grant 772466 ``NOISE'' and by Icelandic Research Fund Grant 185233-051.

\ \\
\ \\
\ \\
\noindent
{Itai Benjamini}\\
Weizmann Institute of Science\\
\texttt{itai.benjamini[at]weizmann.ac.il}\\
\ \\
{\'Ad\'am Tim\'ar}\\
University of Iceland\\
and\\
Alfr\'ed R\'enyi Institute of Mathematics\\
\texttt{madaramit[at]gmail.com}

\end{document}